\documentclass{llncs}

\usepackage[utf8]{inputenc}
\usepackage{enumerate}
\usepackage[shortlabels]{enumitem}
\usepackage{graphicx}
\usepackage{epstopdf}
\usepackage{amsfonts}
\usepackage{amssymb}
\usepackage{amsmath}
\usepackage[labelfont=bf]{caption}
\usepackage{sectsty}
\usepackage{wasysym}
\usepackage{hyperref}
\usepackage{multirow}
\usepackage[usenames,dvipsnames]{xcolor}  

\definecolor{mygreen}{RGB}{243, 255, 198}
\definecolor{codegreen}{rgb}{0,0.6,0}
\definecolor{codegray}{rgb}{0.5,0.5,0.5}
\definecolor{codepurple}{rgb}{0.58,0,0.82}
\definecolor{backcolour}{rgb}{0.95,0.95,0.92}
\colorlet{shadecolor}{mygreen}

\newcommand{\N}{\mathbb{N}}
\newcommand{\Z}{\mathbb{Z}}

\newcommand{\rng}{\mathrm{r}}
\newcommand{\avg}{\mathrm{\overline{r}}}

\newcommand{\graphs}{\mathcal{G}}
\newcommand{\LNR}{Loebl-Nešetřil-Reed\ }
\newcommand{\lip}{\mathcal{L}}

\spnewtheorem{conj}{Conjecture}{\bfseries}{\itshape}

\begin{document}

\title{Graph-indexed random walks\\ on special classes of graphs}

\author{Jan Bok}

\institute{
Computer Science Institute, Charles University in Prague, Malostransk\'{e} n\'{a}m\v{e}st\'{i} 25, 11800, Prague, Czech Republic \email{bok@iuuk.mff.cuni.cz}
}

\date{}

\maketitle

\begin{abstract}
We investigate the paramater of the average range of $M$-Lipschitz mapping
of a given graph. We focus on well-known classes such as paths, complete graphs,
complete bipartite graphs and cycles and show closed formulas for computing
this parameter and also we conclude asymptotics of this parameter on these
aforementioned classes.
\end{abstract}

\section{Introduction}

Graph-indexed random walks (or equivalently also $M$-Lipschitz mapping of
graphs) are a generalization of standard random walk on $\mathbb{Z}$. Also,
this concept has an important connections to statistical physics, namely to
gas models (as is described by Zhao \cite{zhao2016extremal} and Cohen et al.\
\cite{cohen2017widom}). Understanding the structure of all $M$-Lipschitz
mappings of a given graph and corresponding parameters is also a point of
interest because it can describe the expected behavior of a random
homomorphism to a suitable graph.

This paper aims on examining specific classes graphs and its average range.
This parameter can be described, without going into technical details from
the very beginning, as the expected size of the homomorphic image of an
uniformly picked random $M$-Lipschitz mapping of $G$.

Graph-indexed random walks and average range were studied for example in \cite{benjamini2000random,erschler2009random,benjamini2000upper,wuaverage,loebl2003note}.

In the following text we will not mix the terms $M$-Lipschitz mapping and graph-indexed random walk and we will use only the first term.

\subsection{Preliminaries}

In this text we use the standard notation as for example in Diestel's monograph
\cite{diestel2000graph}. To avoid cumbersome notation, we will often write $uv$ for undirected edge.

A \emph{graph homomorphism} between digraphs $G$ and $H$ is a
mapping $f: V(G) \to V(H)$ such that for every edge $uv \in E(G)$, $f(u)f(v)
\in E(H)$. That means that graph homomorphism is an adjacency-preserving
mapping between the vertex sets of two digraphs. The set $I := \{ w \in V(H) \mid
\exists v \in V(G): f(v) = w  \}$ for a graph homomorphism $f$ is
called the \emph{homomorphic image} of $f$.

For a comprehensive and more complete source on graph homomorphisms, the
reader is invited to see \cite{hell2004graphs}. A quick introduction is given
in \cite{godsil2013algebraic} as well.

\begin{definition} \label{def:lipschitz}
For $M \in \N$, an \emph{$M$-Lipschitz mapping} of a connected graph $G = (V,E)$ with root $v_0 \in V$ is a
mapping $f: V \to \Z$ such that $f(v_0) = 0$ and for every edge $(u,v) \in E$
it holds that $|f(u) - f(v)| \le M$. The set of all $M$-Lipschitz mappings of a
graph $G$ is denoted by $\mathcal{L}_M(G)$.
\end{definition}

By the term \emph{Lipschitz mappings of graph} we mean the union of sets of $M$-Lipschitz
mappings for every $M \in N$.

The importance of having rooted graphs is the following. We want to have finitely
many Lipschitz mappings for a fixed graph $G$. Mappings with $f(v_0) \neq 0$
are just linear shifts of some mapping with $f(v_0) = 0$. Formally, consider a mapping
$f'$ with $f'(v_0) = a$. Then we can define a linear transformation $T_a$
as $T_a(f) \longmapsto f - a$. Applying $T_a$ to $f'$ yields a Lipschitz mapping of $G$ with $v_0$ as its root.

We note that we are interested in connected graphs only. Components
without the root would also allow infinitely many new $M$-Lipschitz mappings.

In literature, we will often meet a slightly different definition of
$1$-Lipschitz mappings. In it the restriction $|f(u) - f(v)| \le 1$, for all
$uv \in E$, is removed and instead, the restriction $|f(u) - f(v)| = 1$, for all
$uv \in E$, is added. In \cite{loebl2003note} authors call these mappings
\emph{strong Lipschitz mappings}. We generalize this in the following
definition.

\begin{definition}
For $M \in \N$, a \emph{strong $M$-Lipschitz mapping} of a connected graph $G = (V,E)$ with root $v_0 \in V$ is a
mapping $f: V \to \Z$ such that $f(v_0) = 0$ and for every edge $(u,v) \in E$
it holds that $|f(u) - f(v)| = M$. The set of all $M$-Lipschitz mappings of a
graph $G$ is denoted by $\mathcal{L}_{\pm M}(G)$.
\end{definition}

Note that strong $M$-Lipschitz mappings are a special case of $M$-Lipschitz mappings
of graph. Also, $B$-Lipschitz mappings are a superset of $A$-Lipschitz mapping whenever
$B \geq A$. See Figure \ref{fig:hasse} for the Hasse diagram of various types of Lipschitz mappings.

Analogously, by the term \emph{strong Lipschitz mappings of graph} we mean the union of sets of strong $M$-Lipschitz
mappings for every $M \in N$.

\begin{figure}[h!]
\centering
\includegraphics[scale=0.9]{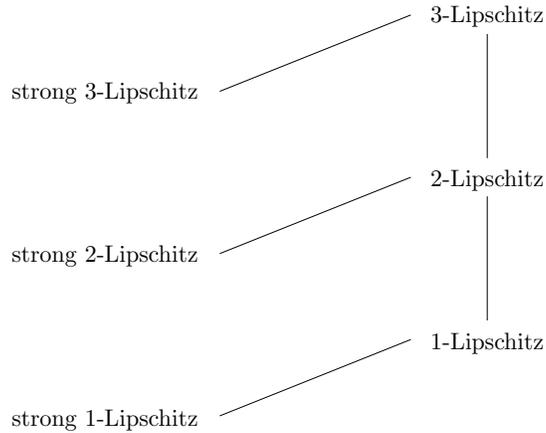}
\caption{The Hasse diagram of different types of Lipschitz mappings of graphs.}
\label{fig:hasse}
\end{figure}

First of all, let us define the \emph{range} of a mapping.

\begin{definition}
The \emph{range} of a Lipschitz mapping $f$ of $G$ is the size of the
homomorphic image of $f$. Formally:
$$\rng_G(f) := \big|\{ z \in \Z\,|\,z = f(v) \textrm{ for some } v \in V(G) \}\big|.$$  
\end{definition}

We define the \emph{average range} of graph $G$ as follows.

\begin{definition} \emph{(Average range)}
The \emph{average range} of graph $G$ over all $M$-Lipschitz mappings is defined as
$$\avg_M(G) := \frac{\sum_{f \in \mathcal{L}_M(G)} \rng(f)}{|\mathcal{L}_M(G)|}.$$
\end{definition}

We can view this quantity as the expected size of the homomorphic image of an
uniformly picked random $M$-Lipschitz mapping of $G$.

Whenever we want to talk about the counterparts of these definitions for
strong Lipschitz mappings, we denote it with $\pm$ in subscript. For example,
$\avg_{\pm M}$ is the average range of strong $M$-Lipschitz mapping of graph.

Whenever we write average range without saying
which $M$-Lipschitz mappings we use, it should be clear from the context
what $M$ do we mean.

It is worth noting that for computing the average range, the choice of
root does not matter. That is why we often omit the details of picking the root.
For better analysis in proofs we occasionally pick the root in some convenient
way.

\subsection{Conjectures on the average range}

Fundamental conjectures on the average range of (strong) $1$-Lipschitz mappings say that paths $P_n$ are extremal
with regard to this parameter on the $n$-vertex graphs.

The first one is from Benjamini, H{\"a}ggstr{\"o}m and Mossel.

\begin{conj} \label{conj:bhm} \cite{benjamini2000random} \emph{(Benjamini-H{\"a}ggstr{\"o}m-Mossel)}
For any connected bipartite graph $G \in \graphs_n$, $\avg_{\pm1}(G) \le \avg_{\pm1}(P_n)$ holds.
\end{conj}

Newer version which generalizes the previous one is the following conjecture
by Loebl, Nešetřil and Reed.

\begin{conj} \label{conj:lnr} \cite{loebl2003note} \emph{(\LNR)}
  For any connected graph $G \in \graphs_n$, $\avg_1(G) \le \avg_1(P_n)$ holds.
\end{conj}

We will occasionally abbreviate Conjecture \ref{conj:bhm} to BHM Conjecture
and Conjecture \ref{conj:lnr} to LNR Conjecture.

\subsection{Structure of this paper}

Each of the following sections deals with a different graph class.

\section{Complete graphs}

For completeness of the picture we will show the formula for $\avg_1(K_n)$.

\begin{theorem} \label{thm:kn}
  For a complete graph $K_n$ we have $\avg_1(K_n) = 2 - (2^n - 1)^{-1}$.
\end{theorem}
\begin{proof}
  Let us count the number of $1$-Lipschitz mappings of $K_n$. We cannot choose the
  image of the root $r$ but we can do it for other vertices. Namely, we must choose
  integers from interval $[-1,1]$ due to the fact that every vertex $v \neq r$
  is a neighbor of the root. Furthermore, $1$-Lipschitz mapping $f$ with
  vertices $u, v \in V(K_n)$ such that $f(u) = -1$ and $f(v) = 1$ cannot exist since
  $uv \in E(K_n)$. Thus, apart from
  the trivial case of setting image of all vertices to $0$, we can choose to
  map vertices other than $r$ either to $-1$ and $0$, or to $1$ and $0$ -- exclusively. For each of this
  choice we have $2^{n-1} - 1$ of such $1$-Lipschitz mappings and each of them has the
  range equal to $2$. We conclude: $$\avg_1(K_n) = \frac{2\cdot 2 \cdot (2^{n-1} - 1) +
  1}{2 \cdot (2^{n-1} - 1) + 1} = 2 - \frac{1}{2^n - 1}.$$
\qed \end{proof}

Theorem \ref{thm:kn} implies the limiting behavior of $\avg_1(K_n)$.

\begin{corollary}
  It holds that $\lim_{n \to \infty} \avg_1(K_n) = 2$.
\end{corollary}

\section{Complete bipartite graphs}

We prove an exact formula for another well-known class of graphs, \emph{complete
bipartite graphs}.

\begin{theorem} \label{thm:kmn} 
  For every $p,q \in \N$, a complete bipartite graph $K_{p,q}$
  satisfies 
  $$|\lip_1(K_{p,q})| = 3^p + 3^q + 2^{p+q} - 2^{p+1} - 2^{q+1} + 1,$$
  and
  $$\avg_1(K_{p,q}) = 3 - 2^{p+q} \cdot (3^p + 3^q + 2^{p+q} - 2^{p+1} - 2^{q+1} + 1)^{-1}.$$
\end{theorem}
\begin{proof}
We use Theorem \ref{thm:diam} that implies that the possible ranges of $K_{p,q}$
form a subset of $\{1,2,3\}$. We analyze separate cases of possible ranges and count
how many such mappings exist. Let us denote the part of size $p$ by $P$ and
the other one, with the size $q$, as $Q$. Without loss of generality, assume that all $1$-Lipschitz
mappings are rooted in some fixed vertex of $P$.

\begin{itemize}
  \item \textit{Range equal to 1:}
    Clearly, there is exactly one such mapping, sending everything to zero.

  \item \textit{Range equal to 2:}
  A homomorphic image of a $1$-Lipschitz mapping is some closed interval, as
  we observed earlier in preliminary chapter. Thus the possibilities for the homomorphic image of the range
  $2$ are $\{0,1\}$ and $\{-1,0\}$. These cases are symmetric to each other, so
  let us analyze, without loss of generality, the case $\{0,1\}$.

  There are $2^{p+q-1}$ possibilities how to assign $0$ and $1$ to the vertices
  excluding the root. However, one of these possibilities is the trivial mapping
  of the range~$1$ (everything mapped to zero). The result is that there are
  $2^{p+q-1}-1$ mappings with the homomorphic image $\{0,1\}$.

  \item \textit{Range equal to 3:}
  Again we have multiple cases. The cases $\{0,1,2\}$ and $\{-2,-1,0\}$ are symmetric,
  the third is $\{-1,0,1\}$.

  Let us solve the case $\{0,1,2\}$ first. Clearly, $0$ and $2$ cannot be in different
  parts, otherwise there would exist an edge with endpoints mapped to $0$ and to $2$,
  violating the definition of $1$-Lipschitz mapping. That further implies the impossibility
  of $v_q \in Q$ mapped to $2$. By a similar argument we get that only in the case
  that all vertices of $Q$ are mapped to one we can get the homomorphic image
  $\{0,1,2\}$. We can then place any of the numbers from $\{0,1,2\}$ on the part $P$. However, we must exclude assignments with no $2$ on the part $P$. That yields
  $3^{p-1}-2^{p-1}$ possibilities.

  The remaining case is $\{-1,0,1\}$. Again, we see that $1$ and $-1$ cannot be in
  different parts. Thus, either the part $P$ has all vertices mapped to zero and
  on $Q$ we can choose for every vertex an image from the set $\{-1,0,1\}$, or vice
  versa. That gives us $3^{p-1}+3^q$ choices from which we must exclude those
  that use only some propper subset of $\{-1,0,1\}$. Finally, we get the formula
  $$3^{p-1}+3^q-2^p-2^{q+1}+2$$
  for this case.
\end{itemize}

Table \ref{table:kmn} summarize all the cases. The number of $1$-Lipschitz mappings of $K_{p,q}$
is equal to $$3^p + 3^q + 2^{p+q} - 2^{p+1} - 2^{q+1} + 1,$$ i.e.\ the sum of the third column of Table \ref{table:kmn}.
By straightforward calculations we get
$$\avg_1(K_{p,q}) = 3 - 2^{p+q} \cdot (3^p + 3^q + 2^{p+q} - 2^{p+1} - 2^{q+1} + 1)^{-1}.$$

\begin{table}[]
\centering
\begin{tabular}{|c|c|c|}
\hline
                range & homomorphic image & number of such mappings \\ \hline\hline
                 1 & $\{0\}$ & 1 \\ \hline
\multirow{2}{*}{2} & $\{0, 1\}$ & $2^{p+q-1}-1$ \\ \cline{2-3} 
                  & $\{0, -1\}$ & dtto \\ \hline
\multirow{3}{*}{3} & $\{0, 1, 2\}$ & $3^{p-1}-2^{p-1}$ \\ \cline{2-3} 
                  & $\{-2, -1, 0\}$ & dtto \\ \cline{2-3} 
                  & $\{-1, 0, 1\}$ & $3^{p-1}+3^q-2^p-2^{q+1}+2$ \\ \hline
\end{tabular}
\caption{Table for Theorem \ref{thm:kmn}.}
\label{table:kmn}
\end{table}
\qed \end{proof}

We conclude this section with the observation on the limiting behavior of
$\avg_1(K_{p,q})$ as $(p+q) \to \infty$. Clearly, the average range is $3$ in
limit.

\section{Stars}

\begin{definition}
  A \emph{star graph} $S_n$ is a tree with $n$ vertices; one vertex of degree
  $n-1$ and $n-1$ leaves (vertices of degree one). Or, alternatively, it is 
  a complete bipartite graph $K_{1,n-1}$.
\end{definition}

\begin{theorem} \label{thm:star}
  A star $S_n$ satisfies
  $$\avg_1(S_n) = 3 - \frac{2^n}{3^{n-1}},$$
  and
  $$\avg_{\pm 1}(S_n) = 3-2^{2-n}.$$
\end{theorem}
\begin{proof}
  We will use the definition of stars as a special case of complete bipartite
  graphs. We can then use Theorem \ref{thm:kmn} for the case of $1$-Lipschitz
  mappings with $p:=1$ and $q:=n-1$. The desired claim follows.

  We will now prove the second formula. Without loss of generality, we will
  root our graph in the central vertex. Observe that all leafs will get either
  $+1$ or $-1$ and only cases in which range is equal to $2$ are the cases
  of either all leaves mapped to $1$ or to $-1$. The rest of the cases
  have the range equal to $3$. Totally, there are $2^{n-1}$ of strong
  $1$-Lipschitz mappings. That concludes our claim.
\qed \end{proof}

\section{Paths}

In \cite{wuaverage}, authors compute several values of $\avg_1(P_n)$ (see
Table \ref{tbl:paths}) and claim
that no explicit formula for an average range of a path is known. We fill this
gap and present such formula, exploiting the tool used in the random walk
analysis called \emph{reflection principle}.

{
\begin{table}[h!]\renewcommand{\arraystretch}{2}\centering
    \begin{tabular}{ | c | c | c | c | c | c | c | c | c | c | c | c |}
    \hline
    $n$ & 2&3&4&5&6&7& 8 & 9 & 10 & 11 & 12 \\ \hline
    $\avg_1(P_n)$ &\large
    $\frac{5}{3}$&\large
    $\frac{19}{9}$&\large
    $\frac{67}{27}$&\large
    $\frac{227}{81}$&\large
    $\frac{751}{243}$ &\large
    $\frac{2445}{729}$ &\large
    $\frac{7869}{2187}$ &\large
    $\frac{25107}{6561}$ &\large
    $\frac{78767}{19683}$ &\large
    $\frac{250793}{59049}$ &\large
    $\frac{786985}{177147}$\\
    \hline
    \end{tabular}
    \caption{Table of values of $\avg_1(P_n)$ for $2 \leq n \leq 12$.}
    \label{tbl:paths}
\end{table}
}

We will define auxiliary random variables and we will speak for a while also
in the language of standard random walks which are naturally encoded in
$1$-Lipschitz mapping of $P_n$. We refer reader to \cite{lovasz1993random}
for a general treatment of random walks.

\begin{definition}
  For a given $1$-Lipschitz mapping $f$:
  \begin{itemize}
  \item $M^+_n(f)$ is a random variable corresponding to the maximum non-negative
  number in the image of a $1$-Lipschitz mapping $f$.
  \item $M^-_n(f)$ is a random variable corresponding to the minimum non-positive
  number in the image of a $1$-Lipschitz mapping $f$.
  \item $X_n(f)$ denotes the number $f(v_n)$, i.e.\ image
  of the second endpoint of $P_n$.
\end{itemize}
We omit $(f)$ if $f$ is clear from the context.
\end{definition} 

\begin{theorem} \label{thm:path}
  For a path $P_n$ we have
  \begin{align*}
  \avg_1(P_n) = 1 + 3^{-n+1} \cdot 2 \cdot \sum_{k=0}^{n-1} k 
  \cdot 
  \sum_{i=0}^{\lfloor \frac{n-1-k}{2} \rfloor} \Bigg( &{n-1 \choose k + i}{n-k-i-1 \choose i}\, + \\&{n-1 \choose k + 1 + i}{n-k-i-2 \choose i}\Bigg).
\end{align*}
\end{theorem} 

\begin{proof}
The average range of path $P_n$ can be formulated as:
$$\avg_1(P_n) = E[M^+_n - M^-_n + 1].$$

From the symmetry of $M^+_n$ and $M^-_n$ and from the linearity of
expectation, one gets:
$$\avg_1(P_n) = E[M^+_n + M^+_n + 1] = E[M^+_n] + E[M^+_n]+ 1 = 2 E[M^+_n] + 1.$$

Set $M_n := M^+_n$. Now let us prove that $P(M_n \geq r) = P(X_n \geq r) + P(X_n
\geq r+1)$.

The walks with $M_n \geq r$ fit into two groups. Either such walks end in $s
\geq r$ or in $s < r$. In the second case, we can reflect the section of the
path after the first time we get to $r$ and we get a new walk which now ends in
$s' > r$. See Figure \ref{fig:reflection} for an illustration. Since this process is
invertible and every path that reaches $s \geq r$ must have $M_n \geq r$, we
get: $$P(M_n \geq r) = P(X_n \geq r) + P(X_n \geq r+1).$$

\begin{figure}[h!]
\centering
\includegraphics[scale=0.8]{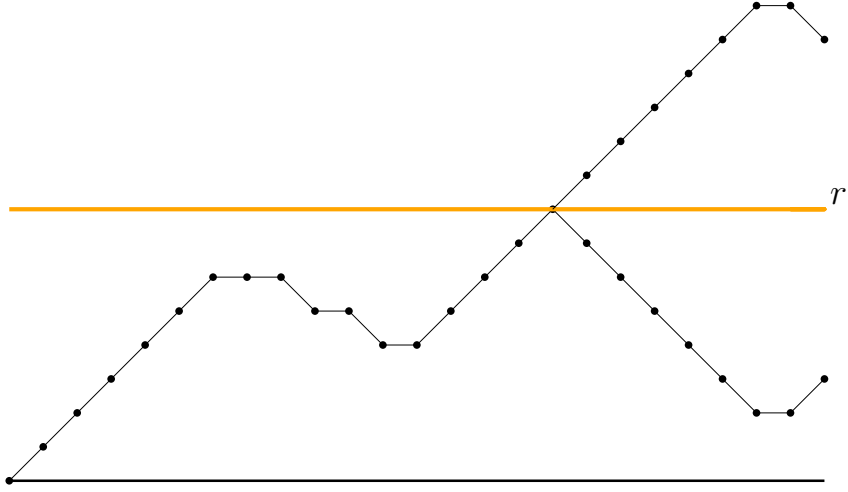}
\caption{An illustration of the reflection principle.}
\label{fig:reflection}
\end{figure}

Next we will prove that: $P(M_n = r) = P(X_n = r) + P(X_n = r+1).$
\begin{align*}
P(M_n = r) &= P(M_n \geq r) - P(M_n \geq r+1) \\
&= P(X_n \geq r) + P(X_n \geq r+1) - P(X_n \geq r+1) - P(X_n \geq r+2) \\
&= P(X_n = r) + P(X_n = r+1).
\end{align*}

Now we need to determine $P(X_n = r)$. Recall the aforementioned bijection
between $\{1,-1,0\}$-sequences and walks from Section \ref{sec:random_walks}.

We have $n-1$ edges so if we want to attain some fixed $k$, we need to sum up
our sequence to $k$. Thus we need to pick $k$ additional $1$'s over $-1$'s.
Summing up through the all possible values of the number of $-1$'s we get:

\begin{equation} \label{eq:first}
  P(X_n = k) = 3^{-n+1} \cdot \sum_{i=0}^{\lfloor \frac{n-1-k}{2} \rfloor} {n-1 \choose k + i}{n-k-i-1 \choose i}.
\end{equation}

And for $P(X_n = k+1)$ analogously:

\begin{equation} \label{eq:second}
  P(X_n = k+1) = 3^{-n+1} \cdot \sum_{i=0}^{\lfloor \frac{n-1-k}{2} \rfloor} {n-1 \choose k + 1 + i}{n-k-i-2 \choose i}.
\end{equation}

We are now ready to combine all of this together and we get:
\begin{align*}
  \avg_1(P_n) &= 1 + 2 \cdot \sum_{k=0}^{n-1} k 
  \cdot \big(P(X_n=k) + P(X_n=k+1) \big).
\end{align*}

We note that ${a \choose b}$ is defined as zero if $b > a$.
Substituting $P(X_n=k)$ by (\ref{eq:first}) and $P(X_n=k+1)$ by
(\ref{eq:second}) we get the desired claim.

\qed \end{proof}

Besides the exact formula for the average range of a path, we prove the
following relation between the $\avg_1$ of paths $P_n$ and $P_{n+1}$.

\begin{lemma} \label{lemma:path}
For every $n \in \N$, $\avg_1(P_{n+1}) - \avg_1(P_n) \le 2/3$.
\end{lemma}
\begin{proof}
  Let us write $v_1,v_2,\ldots, v_n$ for vertices of $P_n$
  consecutively and let $$E(P_n) :=
  \{v_1v_2, v_2v_3, \ldots, v_{n-1}v_n\}.$$

  For $P_{n+1}$, set $V(P_{n+1}) := V(P_n) \cup \{ v_{n+1} \}$ and $E(P_{n+1})
  = E(P_n) \cup \{v_nv_{n+1} \}$.

  Pick $v_1$ as the root of $P_{n}$ and $P_{n+1}$ as well and consider all
  $1$-Lipschitz mappings $\mathcal{L}(P_n)$ and $\mathcal{L}(P_{n+1})$. Choose an
  arbitrary $f$ from $\mathcal{L}(P_n)$. Now $f(v_n) = r$ for some $r \in \Z$.
  If we want to extend this $f$ to a $1$-Lipschitz mapping $f'$ of $P_{n+1}$, we see
  that we can set $f'(v_{n+1})$ to either $r$, $r+1$ or $r-1$. Choosing $r$
  does not increase the range. Since we want to do an upper estimate, let us
  presume that choosing $r+1$ or $r-1$ always increases the range. Thus we
  get: $$ \avg(P_{n+1}) \le \avg(P_n) + 2/3.$$ Which is only a different form
  of the desired claim.
\qed \end{proof}

This simple upper bound has two corollaries.

\begin{corollary}
  For every $r,q \in \N$, $r > q$, $\avg_1(P_r) \le \avg_1(P_q) + (r-q)\cdot\frac{2}{3}$
  holds.
\end{corollary}
\begin{proof}
  Use Lemma \ref{lemma:path} $(r-q)$ times.
\qed \end{proof}

\begin{corollary}
  For every $P_n$, $\avg_1(P_n) \le \frac{2n+1}{3}$ holds.
\end{corollary}
\begin{proof}
  Choose $r:=n$ and $q:=1$. Then use the previous lemma and observe that $\avg_1(P_1)$
  is equal to one.
\qed \end{proof}

We remark that for all paths in general we cannot get a better upper bound by
a constant than in Lemma \ref{lemma:path} since $\avg_1(P_2) - \avg_1(P_1) =
\frac{2}{3}$.

\section{Trees} \label{sec:trees}

The most recent result in the area of graph-indexed random walks is the result
of Wu, Xhu and Zhu from 2016 \cite{wuaverage}. The authors tried to attack the
LNR and BHM conjecture and got the following partial result.

\begin{theorem} \label{thm:trees} \cite{wuaverage}
  For any tree $T_n$ on $n$ vertices holds the following,
  \begin{enumerate}
    \item $\avg_1(T_n) \leq \avg_1(P_n)$,
    \item $\avg_{\pm 1}(T_n) \leq \avg_{\pm 1}(P_n)$.
  \end{enumerate}
\end{theorem}

Their approach is to use a special transformation called KC-transformation,
named by Kelmans \cite{kelmans1981graphs}, which we already mentioned in
Section \ref{sec:sim}. Csikvári \cite{csikvari2010poset,csikvari2013poset}
proved that this transformation induces a partially ordered set on the class of all
$n$-vertex trees with the path $P_n$ as the maximum element and the star $S_n$
as the minimum element. By carefully choosing a right chain in this poset
they prove Theorem~\ref{thm:trees}.

We note that by proving Theorem \ref{thm:star} and Theorem \ref{thm:path}
we showed the precise formulas for the minimum and the maximum possible
average range of trees on $n$ vertices for the case of $1$-Lipschitz mappings.

\section{Cycles}

In this section, more specifically in Theorem \ref{thm:cycle}, we will show a formula for the average range of cycle graphs $C_n$.

See Table \ref{tab:cycle} for values of $\rng_1(C_n)$ of smaller cycles computed
with a help of our computer program.

\begin{table}[h!]\renewcommand{\arraystretch}{2.2}\centering \label{tab:cycle}
    \begin{tabular}{ | c | c | c | c | c | c | c | c | c | c | c |}
    \hline
    $n$ & 3 & 4 & 5 & 6 & 7 & 8 & 9 & 10 & 11 & 12 \\ \hline
    $\avg_1(C_n)$ &\large
    $\frac{13}{7}$ &\large
    $\frac{41}{19}$ &\large
    $\frac{121}{51}$ &\large
    $\frac{365}{141}$ &\large
    $\frac{1093}{393}$ &\large
    $\frac{3281}{1107}$ &\large 
    $\frac{9841}{3139}$ &\large
    $\frac{29525}{8953}$ &\large
    $\frac{88573}{25653}$ &\large
    $\frac{265721}{73789}$ \\
    \hline
    \end{tabular}
    \caption{Table of values of $\avg_1(C_n)$ for $3 \leq n \leq 12$.}
    \label{tab:cycle}
\end{table}

First, let us introduce what the \emph{trinomial triangle} is.

\subsection{Trinomial triangle}

The trinomial triangle is similar to the Pascal (binomial) triangle of binomial
coefficients. One can similarly define trinomial coefficients in a recursive
way.

\begin{definition} \emph{(Trinomial triangle and central trinomial coefficient)}
Trinomial numbers (coefficients) ${n\choose k}_2$ are defined as:
$${0\choose 0}_2=1 $$
$${n+1 \choose k}_2={n \choose k-1}_2+{n \choose k}_2+{n \choose k+1}_2 \textrm{ for } n \geq 0,$$
where ${n\choose k}_2=0$ for $k< -n$ and $k>n$.

\emph{Central trinomial coefficients} are the numbers ${n \choose 0}_2$, where $n \in \N_0$.
\end{definition}

\begin{figure}[h!]
$$\begin{matrix}
 & &  &  &{\color{blue}1}\\
 & &  & 1&{\color{blue}1}&1\\
 & & 1& 2&{\color{blue}3}&2&1\\
 &1& 3& 6&{\color{blue}7}&6&3&1\\
1&4&10&16&{\color{blue}19}&16&10&4&1\end{matrix}$$
\caption{The trinomial triangle with central trinomial coefficients in blue color.}
\label{fig:trinom}
\end{figure}

The sequence for central trinomial coefficients in OEIS is A123456 \cite{A123456}. See Figure~\ref{fig:trinom} for a visualization of the trinomial triangle with
highlighted central trinomial coefficients. Trinomial coefficients appear quite often in enumerative combinatorics.
Let us show one particular example.

\begin{example}
Suppose you have a king on a chessboard (it does not have to be the usual $8 \times 8$ one). Each entry of the triangle corresponds to the number of paths using
the minimum number of steps between some cells of the chessboard. See Figure \ref{fig:king_walks}.
\end{example}

Useful fact is that central trinomial coefficients satisfy the following identity (for its derivation, see for example \cite{blasiak2008motzkin}):
\begin{equation} \label{eq:ctc}
  {n\choose0}_2=\sum_{k=0}^n\frac{n(n-1)\cdots(n-2k+1)}{(k!)^2}=\sum_{k=0}^{\lfloor n/2 \rfloor}{n\choose 2k}{2k\choose k}.
\end{equation}

\begin{figure}[h!]
\centering
\includegraphics[scale=1.0]{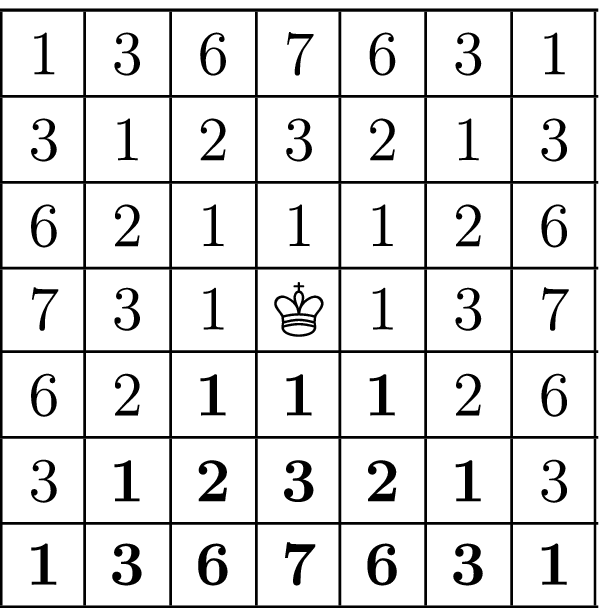}
\caption{Each number represents the number of ways how to get to that cell with
the minimum number of step with the figure of king. \cite{wikichess}}
\label{fig:king_walks}
\end{figure}

\subsection{Motzkin numbers}
For the proof of the formula for $\avg_1(C_n)$ we need to define \emph{generalized Motzkin
number and paths}. We will further write only Motzkin numbers and Motzkin paths.

\begin{definition}
Consider a lattice path, beginning at $(0,0)$, ending at $(n,k)$ and
satisfying that $y$-coordinate of every point is non-negative.
Furthermore, every two consecutive steps $(i,a)$ and $(i+1,b)$
must satisfy $|a-b| \leq 1$. Such lattice path are called Motzkin paths.

The set of all the possible paths ending in $(n,k)$
is denoted by $m(n,k)$ and the cardinality of this set is denoted by
$M(n,k)$. We call $M(n,k)$ the Motzkin number.
\end{definition}

For more details we refer to the seminal paper \cite{donaghey1977motzkin},
Motzkin numbers $M(n,0)$ form the sequence A001006 in OEIS \cite{A001006}. See
Figure \ref{fig:motzkin} for an example of a Motzkin path.

\begin{figure}[h!]
\centering
\includegraphics[scale=1.1]{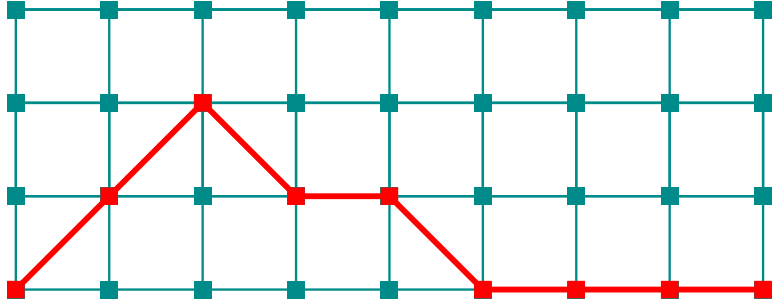}
\caption{A Motzkin path from $(0,0)$ to $(8,0)$.}
\label{fig:motzkin}
\end{figure}

\subsection{Main theorem}

\begin{theorem} \label{thm:cycle}
  For any cycle graph $C_n$, $n \geq 3$, we have
  $$\avg_1(C_n) = \frac{3^n + (-1)^n}{2 \cdot {n \choose 0}_2}.$$
\end{theorem}

We will prove Theorem \ref{thm:cycle} in series of lemmata that will
be put together later.

\begin{lemma} \label{firstbij}
  There is a bijection between $1$-Lipschitz mappings of $C_n$ and
  the set of lattice paths starting at $(0,0)$, ending at $(n,0)$,
  and satisfying that for every two consecutive steps $(i,a)$ and $(i+1,b)$, $|a-b| \leq 1$.
\end{lemma}
\begin{proof}
  The proof is analogous to other bijections we made between $1$-Lipschitz
  mappings of some type and some class of lattice walks. Take the sequence $$v_1,v_2,\ldots, v_n,v_1$$
  of the vertices of $C_n$ such that $v_1$ is the root and the vertices appear
  consecutively on the cycle precisely as in this sequence.
  For every $1$-Lipschitz mapping $f$ of $C_n$ we can define another sequence
  $$(v_1,f(v_1)),(v_2,f(v_2)),\ldots,(v_n,f(v_n)),(v_1,f(v_1)).$$
  The lemma follows easily.
\qed \end{proof}

Let us prove the formula for $|\lip(C_n)|$.

\begin{theorem} \label{thm:cycle}
  For any $C_n$, $n \geq 3$, $|\lip(C_n)| = {n \choose 0}_2$.
\end{theorem}
\begin{proof}
  We will encode all $1$-Lipschitz mappings of the cycle into the sequences
  $\{-1,0,1\}^n$. Consider the lattice walks constructed in Lemma \ref{firstbij}.
  For each sequence
$$(v_1,f(v_1)),(v_2,f(v_2)),\ldots,(v_n,f(v_n)),(v_1,f(v_1)),$$
  one can define the new sequence
  $$f(v_2)-f(v_1),f(v_3)-f(v_2),\ldots,f(v_1)-f(v_n).$$

  We know that these sequences must add up to $0$. Thus for any total number $k$
  of ones in this sequence we must have $k$ times $-1$ in this sequence as well. Furthermore, we have $k
  \le \lfloor n/2 \rfloor$.
  
  Summing over all possible $k$'s we first pick $2k$ edges which
  have either $+1$ or $-1$. Then from these $2k$ edges, we choose $k$ edges for placing
  $1$. The rest of $n-2k$ edges gets $0$'s and the rest of $2k-k=k$ edges gets $(-1)$'s.
  Formally:
  $$\sum_{k=0}^{\lfloor n/2 \rfloor}{n\choose 2k}{2k\choose k}.$$
  This coincides with identity (\ref{eq:ctc}) if we take into account that
  ${a \choose b}$ is defined to be equal to zero if $b > a$.
\qed \end{proof}

\begin{definition}
  We denote by $\mathcal{L}(C_n, -d)$ the set of
  $1$-Lipschitz mappings $f$ of $C_n$ satisfying $$\min_{v \in V(C_n)} f(v) = -d.$$ 
  In other words, $\mathcal{L}(C_n, -d)$ denotes the set of all $1$-Lipschitz mappings
  of $C_n$ with $-d$ as the minimum value in their homomorphic images.
\end{definition}

Another ingredient we need is the following theorem of Van Leeuwen.

\begin{theorem} \cite{van2010some} \label{thm:van}
  Within the class of walks on $Z$ starting at $0$ and with steps advancing
  by $+1$, $0$ or $-1$, there is a bijection, conserving both the length of the 
  walk and the number of steps $0$, between on one hand the walks
  that end in $0$, and on the other hand the walks that do not visit negative
  numbers. The bijection maps walks ending at $0$ and whose minimal number visited
  is $-d$, to walks ending at $2d$, and is realized by reversing the direction
  of the $d$ down-steps that first reach respectively the numbers $-1,-2,\ldots, -d$.
\end{theorem}

Now we need to show a bijection between $\mathcal{L}(C_n,-d)$
and the set $m(n,2d)$.

\begin{lemma} \label{lem:mintomotzkin}

  There exist a bijection from the set of Motzkin paths $m(n,2d)$ to the set 
  $\mathcal{L}(C_n,-d)$.
\end{lemma}
\begin{proof}
   The existence of such sequence follows straightforwardly from combining
   Theorem \ref{thm:van} and Lemma \ref{firstbij}.
\qed \end{proof}

For technical convenience, we will define the \emph{irregular trinomial triangle} 
and \emph{irregular trinomial coefficients}; see the sequence A027907
in OEIS \cite{A027907}. See Figure \ref{fig:irregular_trinomial}, depicting a
part of the irregular trinomial triangle.

\begin{definition}
  The \emph{irregular trinomial coefficients} are defined as
  $$T^*(n,k) = {n \choose k + n}_2.$$
\end{definition}

\begin{table}[]
\centering
\begin{tabular}{c|rrrrrrrrrrrrr}
        n/k  & 0 & 1 & 2 & 3 & 4 & 5 & 6 & 7 & 8 & 9 & 10 & 11 & 12 \\ \hline
        0  & 1 &   &   &   &   &   &   &   &   &   &   &   &  \\
        1  & 1 & 1 & 1  &   &   &   &   &   &   &   &   &   &   \\
        2  & 1  & 2  & 3  & 2  & 1  &   &   &   &   &   &   &   &   \\
        3  & 1  & 3  & 6  & 7  & 6  & 3  & 1  &   &   &   &   &   &   \\
        4  & 1  & 4  & 10  & 16  & 19  & 16  & 10  & 4  & 1  &   &   &   &   \\
        5  & 1  & 5  & 15  & 30  & 45  & 51  & 45  & 30  & 15  & 5  & 1  &   &   \\
        6  & 1  & 6  & 21  & 50  & 90  & 126  & 141  & 126  & 90  & 50  & 21  & 6  & 1  
\end{tabular}
\caption{A part of the irregular trinomial triangle. The entries of the table are numbers $T^*(n,k)$.}
\label{fig:irregular_trinomial}
\end{table}

The following lemmata, showing the relation of Motzkin paths and irregular trinomial coefficients
will be crucial for the proof of the main theorem.

\begin{remark} \label{rem:rec}
  For every $n,k \in \Z$ the following identity holds:
  $$T^*(n,k) = T^*(n-1,k) + T^*(n-1,k-1) + T^*(n-1,k-2).$$
\end{remark}
\begin{proof}
  This is easily verified from the definition of the trinomial coefficients.
\qed \end{proof}

\begin{lemma} \label{lem:diff}
  The following identity holds for every $n,k \in \N_0$, $k \leq n$,
  \begin{equation} \label{eq:motzkin}
    M(n,k) = T^*(n,n-k) - T^*(n, n-k-2).
  \end{equation}
\end{lemma}
\begin{proof}
  We prove this theorem by induction on $n$. For $n = 0,1$, the identity holds.
  
  We divide the rest of the proof into two cases (the first case is needed
  because in case of $n = k$, we would not be able to use induction hypothesis):

  \textit{Case 1: $n = k$.} Then $1 = M(n,n) = T^*(n,0) + T^*(n,-2) = 1 + 0$, so this case is done.

  \textit{Case 2: $n > k$.} Now suppose the identity holds for all numbers up to $n-1$.
  By the definition of the generalized Motzkin numbers we have 
  \begin{align}
    M(n,k) &= M(n-1,k) + M(n-1,k-1) + M(n-1, k+ 1).
  \end{align}
  And by induction hypothesis we can write
  \begin{align*}
    M(n,k) &= T^*(n-1,n-k) - T^*(n-1,n-k-2) \\
             &\quad + T^*(n-1,n-k-1) - T^*(n-1,n-k-3) \\
             &\quad + T^*(n-1,n-k-2) - T^*(n-1,n-k-4).
  \end{align*}
  From Remark \ref{rem:rec} on the recurrence relation of irregular coefficients we
  get that the even summands and odd summands are equal to $T^*(n,n-k)$ and 
  $-T^*(n, n-k-2)$, respectively. Our claim follows.
\qed \end{proof}

Now we need the last lemma, concerning the sum of irregular coefficients.

\begin{lemma} \label{lem:final}
  For every even $n \in N_0$ holds:
  \begin{equation}
    \sum^{n}_{k=0}T^*(n,2k) = (3^n + 1) / 2,
  \end{equation}
  and for odd $n \in N_0$ holds:
  \begin{equation}
    \sum^{n}_{k=1}T^*(n,2k-1) = (3^n - 1) / 2.
  \end{equation}
\end{lemma}
\begin{proof}
  We will prove these identities by induction. For $n = 0,1$, the respective identities hold.
  Now assume that both identities hold for all $n' < n$. By parity of $n$ we distinguish two
  cases. We will prove the lemma for the case of $n$ even. Odd case is very similar.
    \begin{align}
      \sum^{n}_{k=0}T^*(n,2k) &= \sum^{n-1}_{k=1}T^*(n,2k-1) + 2 \cdot \sum^{n-1}_{k=0}T^*(n,2k) \\
      &= 2\cdot 3^{n-1} - \sum^{n-1}_{k=1}T^*(n,2k-1) \\
      &= 2\cdot 3^{n-1} - (3^{n-1} - 1) / 2 \\
      &= (3^n + 1) / 2.
    \end{align}
    \begin{itemize}
      \item The first equation follows from Remark \ref{rem:rec}.
      \item The second equation follows from the fact that the sum of the $n$-th row of $T^*$ is equal
    to $3^{n}$. That can be easily proved by induction.
      \item The third equation follows from the induction hypothesis.
      \item The fourth equation is straightforward calculation.
    \end{itemize}
\qed \end{proof}

We can finally prove the main theorem of this section and one of the main
results of this paper.

\begin{proof}[Proof of Theorem \ref{thm:cycle}]
  We will first show the following identity for every $n \geq 3$.
\[ \label{eq:telescope}
    \sum_{k=0}^{\lfloor n/2 \rfloor} (2k+1) (T^*(n,n-2k) - T^*(n-1, n-2k-2))= 
\begin{cases}
    \sum^{n}_{k=0}T^*(n,2k),&n \text{ even}\\
    \sum^{n}_{k=1}T^*(n,2k-1),&n \text{ odd}
\end{cases}
\]
This identity follows from the straightforward calculations and from the observation
that $T^*(n,n-k) = T^*(n,n+k)$ for every $n,k \in Z$.

  For brevity, we will do the following calculation for $n$ odd. The proof for $n$ even is
  different in the last two equations but the only difference is the use of the different
  parts of lemma and identity \ref{eq:telescope}, depending on the parity.

  \begin{align}
  \avg_1(C_n) \cdot |\lip(C_n)| &= \sum_{k=0}^{\lfloor n/2 \rfloor} (2k+1)\cdot M(n,2k) \tag{by Lemma \ref{lem:mintomotzkin} and linearity of expectation} \\
  &= \sum_{k=0}^{\lfloor n/2 \rfloor} (2k+1)\cdot \big(T^*(n,n-2k) - T^*(n-1, n-2k-2)\big) \tag{by Lemma \ref{lem:diff}} \\
      & = \sum^{n}_{k=1}T^*(n,2k-1) \notag \\
      & = \frac{3^n -1}{2} \tag{by Lemma \ref{lem:final}}.
\end{align}

  Together with Theorem \ref{thm:cycle} taken into account we conclude the formula for $\avg_1(C_n)$.
\qed \end{proof}

We present the following corollary regarding the asymptotics of $\avg_1(C_n)$.

\begin{corollary}
  It holds that $\avg_1(C_n) \sim 2\sqrt{\frac{\pi}{3}n}$.
\end{corollary}
\begin{proof}
  The asymptotics of central trinomial coefficients is known, see e.g.\ \cite[p. 588]{flajolet2009analytic}.
  Central trinomial coefficients satisfy
  $${n \choose 0}_2 \sim \frac{3^{n+1/2}}{2\sqrt{\pi n}}.$$
  The sign $\sim$ denotes the relation of two sequences. Two sequences $a_n$ and $b_n$
  are in relation $a_n \sim b_n$ if $\lim_{n \to \infty} \frac{a_n}{b_n} = 1$.
  Using Theorem \ref{thm:cycle} and the mentioned asymptotics, we get:
  $\avg_1(C_n) \sim 2\sqrt{\frac{\pi}{3}n}.$
\qed \end{proof}

\section{Pseudotrees} \label{sec:pseudotrees}

We suspect that the following results might be the first step to prove LNR and BHM conjectures for the class of pseudotrees.

\begin{definition}
  We call a graph \emph{unicyclic} if it contains exactly one cycle.
\end{definition}

\begin{definition}
   We call a graph \emph{pseudotree} if it is a tree or a unicyclic graph.
  Equivalently, pseudotrees are graphs with at most one cycle. 
\end{definition}

\subsection{Counting the number of $1$-Lipschitz mappings}

\begin{lemma} \label{lem:number_unicyclic}
  The number of $1$-Lipschitz mappings of unicyclic graphs with order $n$ and cycle size
  $c$, $c \leq n$, is equal to
  $${c\choose 0}_2 \cdot 3^{n-c}.$$
\end{lemma}
\begin{proof}
  Let us denote our unicyclic graph of order $n$ and cycle size $c$ by $G$ and the subgraph induced
  by the vertices on its cycle by $C$.
  We use Theorem \ref{thm:cycle} to get the number of $1$-Lipschitz mappings of
  the subgraph $C$. Now let us fix some $f$, a~$1$-Lipschitz mapping of $C$.

  By deleting all the edges of the cycle $C$ we get a forest $\mathcal T$ of trees $T_1,\ldots,T_c$.
  In this forest, exactly one vertex in each tree $T_i$ has an image under the mapping $f$. Thus
  we obtain $3^{|V(T_i)|-1}$ different $1$-Lipschitz mappings for each
  of the tree in $\mathcal T$. Because we can choose all these mapping independently
  on each other, we obtain, summing over all possible mappings $f$, the following identity.
  $$\lip_1(G) = {c\choose 0}_2 \cdot 3^{\,\sum_{i=1}^c |V(T_i)|-1} = {c\choose 0}_2 \cdot 3^{n-c}.$$
\qed \end{proof}

Observe that Lemma \ref{lem:number_unicyclic} implies that two same-order unicyclic
graphs with the same-length cycle have the same number of Lipschitz mappings.

\subsection{KC-transformation}

In this section it will be useful for us to give a name to
one special subset of unicyclic graphs. See Figure \ref{fig:corolla} for an example.

\begin{definition}
  A \emph{corolla graph} is a unicyclic graph obtained by taking a cycle graph
  and joining some path graphs to it by identifying their endpoints with some vertex of that cycle.
  Every path is joined to exactly one vertex of the cycle. And every vertex
  of the cycle has at most one path attached.
\end{definition}

We note that cycles form a subset of corolla graphs.

\begin{figure}[h!]
\centering
\includegraphics[scale=0.7]{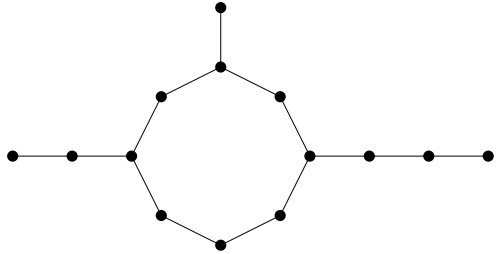}
\caption{An example of corolla graph.}
\label{fig:corolla}
\end{figure}

Now we are ready to introduce the \emph{generalized KC-transformation} and
the main result of \cite{wuaverage}.

\begin{definition}[Generalized KC-transformation]
Take a connected graph $G$ and pick $\{a,b\} \in {V(G) \choose 2}$.
Let $V_{a;b}(G)$ denote the set of those vertices which cannot reach $b$ without
passing by $a$ in $G$. If it is satisfied the following condition that
$$\min(|V_{a;b}(G)|,|V_{b;a}(G)|) > 1,$$
then we can get a new graph $G_{a \to b}$ by modifying $G$ in the following way.

Remove the edges $bb_1,\ldots,bb_t$, where $b_1,\ldots, b_t$ are all the neighbors of
$b$ in $V_{b;a}(G)$ and add new edges $ab_1,\ldots,ab_t$.
\end{definition}


\begin{definition}
  Let $G$ be a connected graph. Take two different cut vertices $a$
  and $b$ of $G$. We write $V(G;a,b)$ for the set
  $$\big(V(G) \setminus (V_{a;b}(G) \cup V_{b;a}(G)) \big) \cup \{a,b\}.$$
\end{definition}

\begin{theorem} \cite{wuaverage} \label{thm:cesty}
Let $G$ be a connected graph. Take two different cut vertices $a$ and $b$ of $G$.
Let $H$ be the subgraph of $G$ induced by $V(G;a,b)$.
Assume that $H$ has an automorphism $\sigma$ such that $\sigma(a)=b$ and
$\sigma(b)=a$. Then $\avg_1(G) \ge \avg_1(G_{a \to b})$.
\end{theorem}

It is worth noting that one of the corollaries of Theorem \ref{thm:cesty} is the
aforementioned Theorem \ref{thm:trees}. We will use Theorem \ref{thm:cesty} to
show that for every unicyclic graph that is not a corolla graph there exists some
corolla graph of the same order and cycle size that has higher or equal $\avg_1$.

\begin{theorem}
  For every unicyclic graph $U$ on $n$ vertices that is not a corolla graph
  there exist a corolla graph $R$ on $n$ vertices such that
  $\avg_1(R) \geq \avg_1(U)$.
\end{theorem}
\begin{proof}
  Take an inclusion-wise maximal tree $T$ rooted in $r$ such that
  $r$ is a vertex of the cycle
  of $U$ and $T$ is not isomorphic to a path graph. Furthermore, $T$ must
  satisfy $V(U) \cap V(T) = \{r\}$ . Since $U$ is not a corolla graph,
  such tree must exist.

  Consider a sequence $T_1,T_2,\ldots,T_s$ with $T_1 = T$ and $T_s$ being
  a path graph such that for every $i \in \{2,\ldots, s\}$,
  $\avg_1(T_{i-1}) \leq \avg_1(T_i)$ holds. The existence of such sequence directly follows
  from Theorem \ref{thm:cesty}.

  We can easily extend this argument and define the sequence
  $U_1,U_2,\ldots,U_s$ such that $U_i$ is the graph in which $T$ is replaced
  by $T_i$. Clearly, $\avg_1(U_{i-1}) \leq \avg_1(U_i)$.

  We can repeatedly find another tree $T'$ in $U_s$, satisfying the same
  conditions as $T$ (except that the root has to be of different of course) in $U$ and proceed similarly until we cannot find
  some next $T'$. We get a corolla graph and our claim follows.
\qed \end{proof}

\section{Concluding remarks}

We have showed closed formulas for several classes of graphs including paths,
complete graphs, complete bipartite graphs (and specially stars) and most
importantly we showed the formula for cycles by using properties of
generalized Motzkin numbers. We also investigated pseudotrees in the effort
of extending the results of \cite{wuaverage}.

\section*{Acknowledgments}

This research was supported by the Charles University Grant Agency, project GA UK 1158216.

\bibliographystyle{acm}
\bibliography{bibliography}

\end{document}